\documentclass[12pt]{amsart}

\usepackage{latexsym,amssymb,amsmath, amsfonts}
\usepackage{graphicx}
\usepackage{pifont}
\usepackage{mathrsfs}
\usepackage{epsfig,verbatim}
\usepackage{color}

\numberwithin{equation}{section}
\newtheorem{theorem}{Theorem}[section]
\newtheorem{lemma}[theorem]{Lemma}

\newtheorem{proposition}[theorem]{Proposition}
\newtheorem{definition}[theorem]{Definition}
\newtheorem{rem}[theorem]{Remark}
\newtheorem{Rem}{Remark}[section]
\newenvironment{remark}{\begin{Rem}\rm}{\end{Rem}}

\newcommand\BbR{\mathbb{R}}
\newcommand\R{\mathbb{R}}

\newcommand\N{\mathbb{N}}
\newcommand\BbZ{\mathbb{Z}}

\newcommand\be{\begin{equation}}
\newcommand\ee{\end{equation}}
\newcommand\bea{\begin{eqnarray}}
\newcommand\eea{\end{eqnarray}}
\newcommand\beaa{\begin{eqnarray*}}
\newcommand\eeaa{\end{eqnarray*}}
\newcommand\bR{{\Bbb R}}

\renewcommand{\epsilon}{\varepsilon}

\begin{document}

\title[traveling wave for a lattice dynamical system]
{Traveling waves for a lattice dynamical system arising in a diffusive endemic model}

\author[Y.-Y. Chen]{Yan-Yu Chen}
\address[Y.-Y. Chen]{Department of Mathematics, Tamkang University, Tamsui, New Taipei City, Taiwan}
\email{chenyanyu24@gmail.com}

\author[J.-S. Guo]{Jong-Shenq Guo}
\address[J.-S. Guo]{Department of Mathematics, Tamkang University, Tamsui, New Taipei City, Taiwan}
\email{jsguo@mail.tku.edu.tw}

\author[F. Hamel]{Fran{\c{c}}ois Hamel}
\address[F. Hamel]{Aix Marseille Universit\'e, CNRS, Centrale Marseille, Institut de Math\'ematiques de Marseille, UMR 7373, 13453 Marseille, France}
\email{francois.hamel@univ-amu.fr}

\thanks{Date: \today. Corresponding Author: J.-S. Guo}

\thanks{This work was supported in part by the Ministry of Science and Technology of the Republic of China under the grants 104-2811-M-032-005 and 102-2115-M-032-003-MY3. This work has been carried out in the framework of the Labex Archim\`ede (ANR-11-LABX-0033) and of the A*MIDEX project (ANR-11-IDEX-0001-02), funded by the ``Investissements d'Avenir" French Government program managed by the French National Research Agency (ANR). The research leading to these results has received funding from the European Research Council under the European Union's Seventh Framework Programme (FP/2007-2013) / ERC Grant Agreement n.321186 - ReaDi - Reaction-Diffusion Equations, Propagation and Modelling, and from the French ANR within the project NONLOCAL (ANR-14-CE25-0013).}



\begin{abstract}
This paper is concerned with a lattice dynamical system modeling the evolution of susceptible and infective individuals at discrete niches. We prove the existence of traveling waves connecting the disease-free state to non-trivial leftover concentrations. We also characterize the minimal speed of traveling waves and we prove the non-existence of waves with smaller speeds.
\end{abstract}

\maketitle


\section{Introduction}\label{intro}
\setcounter{equation}{0}

In this article, we consider the following lattice dynamical system (LDS)
\bea\label{SI}
  \left\{
  \begin{array}{rcl}
\dfrac{ds_n}{dt}&=&(s_{n+1}+s_{n-1}-2s_n)+\mu-\mu\,s_n-\beta\,s_n\,i_n,\; n\in \mathbb{Z},\vspace{3mm}\\
\dfrac{di_n}{dt}&=&d(i_{n+1}+i_{n-1}-2i_n)-\mu\,i_n+\beta\,s_n\,i_n-\gamma\,i_n,\; n\in \mathbb{Z},
  \end{array}
  \right.
\eea
where $s_n=s_n(t)$, $i_n=i_n(t)$, $t\in\BbR$, and $\mu,\beta,\gamma$ are positive constants. Here $s_n(t)$ and $i_n(t)$ represent the population density of the susceptible individuals and the infective individuals at niches $n$ at time $t$, $1$ and $d$ are the random migration coefficients for susceptible and infective population, respectively, and $\mu$ is regarded as the rate of the inflow of newborns into the susceptible population by assuming the total population of susceptible, infective and recovered individuals is normalized to be~$1$. The death rate of the susceptible population and the infective population are both assumed to be $\mu$, $\beta$ is the infective (transmission) coefficient and $\gamma$ is the recovered/removed coefficient. Actually, as in~\cite{H}, the equation for the recovered individuals $r_n(t)$ is given by
$$\dfrac{dr_n}{dt}=\gamma i_n-\mu r_n,$$
if there is no migration of the recovered individuals.

For system~\eqref{SI}, it easy to see there are two constant states $(1,0)$ and
\be\label{endemic}
(s^*,e^*):=\Big(\frac{1}{\sigma},\frac{\mu}{\beta}\,(\sigma-1)\Big),
\ee
where $\sigma:=\beta/(\mu+\gamma)$. In this paper, we always assume that $\sigma>1$, that is,
$$\beta>\mu+\gamma.$$
This means that, when the density of susceptible individuals is close to~$1$, the infective individuals have a positive per capita growth rate. Without migration, the steady state $(1,0)$ is dynamically unstable with respect to perturbations whose second component are positive, while the steady state $(s^*,e^*)$ is dynamically stable. System~\eqref{SI} is therefore called monostable.

In this paper, we are interested in the existence of traveling wave solutions of~\eqref{SI}. Here a traveling wave solution is a bounded solution which can be expressed by
$$s_n(t)=\phi(n+ct)\ \hbox{ and }\ i_n(t)=\psi(n+ct)$$
for $n\in\BbZ$ and $t\in\BbR$, for some nonnegative bounded functions $\phi, \psi$ on $\BbR$ (the wave profiles) and some constant $c$ (the wave speed). By setting $\xi=n+ct$ and substituting $(s_n(t),i_n(t))=(\phi(\xi),\psi(\xi))$ into~\eqref{SI}, we obtain
\be\label{eq:TW}
\left\{\begin{array}{rcl}
-c\,\phi'(\xi)+D[\phi](\xi)+\mu\,(1-\phi(\xi))-\beta\,\phi(\xi)\,\psi(\xi) & = & 0,\vspace{3pt}\\
-c\,\psi'(\xi)+d\,D[\psi](\xi)-(\mu+\gamma)\,\psi(\xi)+\beta\,\phi(\xi)\,\psi(\xi) & = & 0
\end{array}\right.
\ee
for all $\xi\in\R$, where
$$D[f](\xi):=f(\xi+1)+f(\xi-1)-2f(\xi).$$
Furthermore, from the epidemic point of view, we are interested in traveling wave solutions connecting the trivial disease-free state $(1,0)$ as $\xi\to-\infty$ (ahead of the front) and non-trivial states as $\xi\to+\infty$.

Define the constant $c^*$ by
\be\label{c*}
c^*:=\inf_{\lambda>0}\frac{d\,(e^{\lambda}+e^{-\lambda}-2)+\beta-\mu-\gamma}{\lambda}.
\ee
By the assumption $\beta>\mu+\gamma$ (i.e. $\sigma>1$), we know that $c^*=c^*(d, \beta, \mu, \gamma)$ is a well-defined real number (and the infimum in~\eqref{c*} is a minimum) and $c^*>0$.

Our main result is the following theorem on the existence of traveling waves for~\eqref{eq:TW} and the characterization of their minimal speed.

\begin{theorem}\label{th:exist}
For any $c\ge c^*$, there exists a bounded classical solution $(\phi, \psi)$ of the system~\eqref{eq:TW} such that
\be\label{nontrivial}
0<\phi<1\hbox{ in }\BbR,\ \ \psi>0\hbox{ in }\BbR
\ee
and
\be\label{bc-minus}
\lim_{\xi\rightarrow-\infty}(\phi(\xi),\psi(\xi))=(1,0),
\ee
together with
\be\label{bc-plus}
0\!<\!\liminf_{\xi\rightarrow+\infty}\phi(\xi)\!\le\!s^*\le\limsup_{\xi\rightarrow+\infty}\phi(\xi)\!<\!1\ \hbox{ and }\ 0\!<\!\liminf_{\xi\rightarrow+\infty}\psi(\xi)\!\le\!e^*\!\le\!\limsup_{\xi\rightarrow+\infty}\psi(\xi)\!<\!+\infty,
\ee
where $(s^*,e^*)$ is given in~\eqref{endemic}. Furthermore, for any $c<c^*$, there is no classical solution $(\phi, \psi)$ of the system~\eqref{eq:TW} satisfying~\eqref{nontrivial} and~\eqref{bc-minus}.
\end{theorem}

Behind the front, as $\xi\to+\infty$, the leftover concentrations of susceptible and infective individuals are non-trivial. It is still an open question to know whether the traveling wave solutions converge to the endemic state $(s^*,e^*)$ as $\xi\to+\infty$, but Theorem~\ref{th:exist} asserts that both susceptible and infective individuals coexist behind the front and that the endemic state $(s^*,e^*)$ is the only possible constant leftover state. To show the convergence to the endemic state as $\xi\to+\infty$, the difficulties come from the fact that~\eqref{eq:TW} is a system and is non-local (such issues also arise for equations with non-local nonlinear interaction, see e.g.~\cite{AC,ABVV,BNPR,FZ,FH,GL,GVA,NPT,NRRP}). We also point out that the random migration coefficient $d$ for the infective individuals is any arbitrary positive real number and is therefore in general different from that for the susceptible individuals. Furthermore, we mention that, due to the transmission term
 s with opposite signs, the systems~\eqref{SI} and~\eqref{eq:TW} are not monotone (ne
 ither cooperative nor competitive) and therefore do not satisfy the maximum principle.

Let us finally mention some references on related problems. Actually, there is a vast literature on the study of traveling wave solutions for lattice dynamical systems or discrete versions of continuous parabolic partial differential equations. For monostable equations or monostable monotone systems, we refer to e.g.~\cite{CFG,CG1,CG2,FWZ1,FWZ2,GH,GW1,GW3,HK,HZ2,MWZ,MZ,XM,ZHH}. Waves for bistable lattice dynamical systems have been studied in e.g.~\cite{CMV,CGW,CMS,CF1,CF2,CF3,CF4,GW2,HZ,HHV,K1,K2,LW,Z1,Z2}.

\begin{rem}{\rm
Notice that the necessity condition $c\ge c^*$ holds for any traveling wave $(\phi,\psi)$ satisfying~\eqref{eq:TW},~\eqref{nontrivial} and~\eqref{bc-minus}. The limiting conditions~\eqref{bc-plus} or the boundedness of $\psi$ are not used here.}
\end{rem}

\noindent{\bf{Outline of the paper.}} Sections~\ref{sec2} and~\ref{sec3} are devoted to the proof of the existence of a traveling wave in case $c>c^*$, with some preliminaries on lower and upper solutions in Section~\ref{sec2}. Approximated solutions in bounded domains are constructed and the traveling wave solving~\eqref{eq:TW} is obtained by passing to the limit in the whole real line. Some intricate issues are to show that the limiting $\psi$ component is bounded and that the leftover concentrations are non-trivial. Section~\ref{secc*} is devoted to the proof of the existence of a traveling wave for the minimal speed $c^*$, by passing to the limit $c_k\to(c^*)^+$, after a suitable shift of the origin and after showing that the solutions with speed $c_k$ are uniformly bounded. Lastly, Section~\ref{sec5} is concerned with the proof of the necessity condition $c\ge c^*$ for any traveling wave satisfying~\eqref{nontrivial} and~\eqref{bc-minus}.


\section{Preliminaries}\label{sec2}
\setcounter{equation}{0}

In this section, we always assume that $c>c^*$. Then the equation
\be\label{psi-char1}
d\,(e^{\lambda}+e^{-\lambda}-2)-c\,\lambda+\beta-\mu-\gamma=0
\ee
has two positive roots $\lambda_1$ and $\lambda_2$ with $0<\lambda_1<\lambda_2$. Notice that
$$d\,(e^{\lambda}+e^{-\lambda}-2)-c\,\lambda+\beta-\mu-\gamma<0$$
for all $\lambda\in(\lambda_1,\lambda_2)$.


\subsection{Upper and lower solutions}

First, we define the notion of upper solution $(\overline{\phi}, \overline{\psi})$ and lower solution $(\underline{\phi}, \underline{\psi})$ of~\eqref{eq:TW} as follows.

\begin{definition}
If $\overline{\phi}$, $\overline{\psi}$, $\underline{\phi}$, $\underline{\psi}$ are continuous in $\R$, of class $C^1$ on $\BbR\setminus \mathcal{F}$ for some finite set $\mathcal{F}$ and if they satisfy the following inequalities
\bea
&&D[\overline{\phi}](\xi)-c\,\overline{\phi}'(\xi)+\mu\,(1-\overline{\phi}(\xi))-\beta\,\overline{\phi}(\xi)\,\underline{\psi}(\xi)\le0,\label{i1}\\
&&D[\underline{\phi}](\xi)-c\,\underline{\phi}'(\xi)+\mu\,(1-\underline{\phi}(\xi))-\beta\,\underline{\phi}(\xi)\,\overline{\psi}(\xi)\ge0,\label{i2}\\
&&d\,D[\overline{\psi}](\xi)-c\,\overline{\psi}'(\xi)-(\mu+\gamma)\,\overline{\psi}(\xi)+\beta\,\overline{\phi}(\xi)\,\overline{\psi}(\xi)\le0,\label{i3}\\
&&d\,D[\underline{\psi}](\xi)-c\,\underline{\psi}'(\xi)-(\mu+\gamma)\,\underline{\psi}(\xi)+\beta\,\underline{\phi}(\xi)\,\underline{\psi}(\xi)\ge0\label{i4}
\eea
for all $\xi\in\BbR\setminus \mathcal{F}$, then the functions $(\overline{\phi}, \overline{\psi})$, $(\underline{\phi}, \underline{\psi})$ are called a pair of upper and lower solutions of~\eqref{eq:TW}.
\end{definition}

Following~\cite{BHKR,FGW}, we introduce
\bea\label{v2}
\overline{\phi}(\xi)&=&1,\quad \overline{\psi}(\xi)=e^{\lambda_1\xi},\quad\ \ \xi\in\BbR,\label{upper}\\
\underline{\phi}(\xi)&=&\left\{\begin{array}{rll}
&1-\rho\,e^{\theta\xi},& \quad\ \ \ \xi\le \xi_1,\vspace{3mm}\\
&0,& \quad\ \ \ \xi\ge \xi_1,\end{array}\right.\label{u2}\\
\underline{\psi}(\xi)&=&\left\{\begin{array}{rll}
&e^{\lambda_1\xi}-q\,e^{\eta\lambda_1\xi},& \xi\le \xi_2,\vspace{3mm}\\
&0,& \xi\ge \xi_2,\end{array}\right.\label{underpsi}
\eea
where
\be\label{xi12}
\xi_1:=-\frac{\ln\rho}{\theta}\ \hbox{ and }\ \xi_2:=-\frac{\ln q}{(\eta-1)\,\lambda_1}.
\ee
Here the constants $\theta$, $\rho$, $\eta$ and $q$ are chosen {\it in sequence} such that the following assumptions (A1)-(A4) hold:
\begin{enumerate}
\item[(A1)]$\theta>0$ is small enough such that $0<\theta<\lambda_1$ and $e^{\theta}+e^{-\theta}-2-c\,\theta-\mu<0$,
\item[(A2)]$\rho>\max\left\{1, \dfrac{\beta}{-(e^{\theta}+e^{-\theta}-2-c\,\theta-\mu)}\right\}\ge1>0$,
\item[(A3)]$\eta\in(1, \min\{1+\theta/\lambda_1, \lambda_2/\lambda_1\})$ such that
$$d\,(e^{\eta\lambda_1}+e^{-\eta\lambda_1}-2)-c\,\eta\,\lambda_1+\beta-\mu-\gamma<0,$$
\item[(A4)]$q>\max\left\{e^{(1-\eta)\lambda_1\xi_1}, \dfrac{\beta\,\rho}{-(d\,(e^{\eta\lambda_1}+e^{-\eta\lambda_1}-2)-c\,\eta\,\lambda_1+\beta-\mu-\gamma)}\right\}>0$.
\end{enumerate}

Note that we have
$$\xi_2=-\dfrac{\ln q}{(\eta-1)\lambda_1}<\xi_1=-\dfrac{\ln\rho}{\theta}<0.$$
Also, it easy to see that
$$\max\{0,1-\rho\,e^{\theta\xi}\}\le\underline{\phi}(\xi)\le\overline{\phi}(\xi)=1,\quad\max\{0,e^{\lambda_1\xi}-q\,e^{\eta\lambda_1\xi}\}\le\underline{\psi}(\xi)\le\overline{\psi}(\xi)=e^{\lambda_1\xi}$$
for all $\xi\in\BbR$.

The next lemma gives the existence of a pair of upper and lower solutions.

\begin{lemma}\label{lem:UL}
The functions $(\overline{\phi}, \overline{\psi})$ and $(\underline{\phi}, \underline{\psi})$ defined by~\eqref{upper}-\eqref{underpsi} are a pair of upper and lower solutions of~\eqref{eq:TW}.
\end{lemma}

\begin{proof}
First, the functions $\overline{\phi}$ and $\overline{\psi}$ are of class $C^1(\R)$ and the inequalities~\eqref{i1} and~\eqref{i3} hold on $\BbR$, since
$$\left\{\begin{array}{l}
D[\overline{\phi}](\xi)-c\,\overline{\phi}'(\xi)+\mu\,(1-\overline{\phi}(\xi))-\beta\,\overline{\phi}(\xi)\,\underline{\psi}(\xi)=-\beta\,\underline{\psi}(\xi)\le0,\vspace{3pt}\\
d\,D[\overline{\psi}](\xi)-c\,\overline{\psi}'(\xi)-(\mu+\gamma)\,\overline{\psi}(\xi)+\beta\,\overline{\phi}(\xi)\,\overline{\psi}(\xi)\\
\quad =e^{\lambda_1\xi}\,\big[d\,(e^{\lambda_1}+e^{-\lambda_1}-2)-c\,\lambda_1+\beta-\mu-\gamma\big]=0
\end{array}\right.$$
for all $\xi\in\BbR$.

Next, the function $\underline{\phi}$ is continuous in $\R$ and of class $C^1(\R\backslash\{\xi_1\})$ and we would like to show that~\eqref{i2} holds for $\xi\neq\xi_1$. For $\xi>\xi_1$, this is trivial since $\underline{\phi}(\xi)=0$. When $\xi<\xi_1\,(<0)$, we have $\underline{\phi}(\xi)=1-\rho e^{\theta\xi}$ and so
\beaa
&&D[\underline{\phi}](\xi)-c\,\underline{\phi}'(\xi)+\mu\,(1-\underline{\phi}(\xi))-\beta\,\underline{\phi}(\xi)\,\overline{\psi}(\xi)\\
&=&1-\rho\,e^{\theta(\xi+1)}+1-\rho\,e^{\theta(\xi-1)}-2+2\,\rho\,e^{\theta\xi}+c\,\theta\,\rho\,e^{\theta\xi}+\mu\,\rho\,e^{\theta\xi}-\beta\,e^{\lambda_1\xi}+\beta\,\rho\,e^{(\theta+\lambda_1)\xi}\\
&\ge&e^{\theta\xi}\,\big[-\rho\,(e^{\theta}+e^{\theta}-2\,c\,\theta-\mu)-\beta\,e^{(\lambda_1-\theta)\xi}\big]\\
&\ge&\beta\,e^{\theta\xi}\big[1-e^{(\lambda_1-\theta)\xi}\big]\ge0
\eeaa
by $\theta<\lambda_1$ and the choice of $\rho$.

Finally, the function $\underline{\psi}$ is continuous in $\R$ and of class $C^1(\R\backslash\{\xi_2\})$ and we claim that~\eqref{i4} holds for $\xi\neq\xi_2$. Clearly,~\eqref{i4} holds for $\xi>\xi_2$. For the case $\xi<\xi_2$, due to $\xi_2<\xi_1<0$, we know that $\underline{\phi}(\xi)=1-\rho e^{\theta\xi}$ and $\underline{\psi}(\xi)=e^{\lambda_1\xi}-qe^{\eta\lambda_1\xi}$. Then we obtain
\beaa
&\!\!&d\,D[\underline{\psi}](\xi)-c\,\underline{\psi}'(\xi)-(\mu+\gamma)\,\underline{\psi}(\xi)+\beta\,\underline{\phi}(\xi)\,\underline{\psi}(\xi)\\
&\!\!\ge\!\!&d\,\big[\!-q\,e^{\eta\lambda_1(\xi+1)}\!-q\,e^{\eta\lambda_1(\xi-1)}\!+2\,q\,e^{\eta\lambda_1\xi}\big]\!+c\,\eta\,\lambda_1\,q\,e^{\eta\lambda_1\xi}\!-(\beta-\mu-\gamma)\,q\,e^{\eta\lambda_1\xi}\!-\beta\,\rho\,e^{(\theta+\lambda_1)\xi}\\
&\!\!=\!\!&e^{\eta\lambda_1\xi}\Big\{\!-q\big[d\,(e^{\eta\lambda_1}+e^{-\eta\lambda_1}-2)-c\,\eta\,\lambda_1+\beta-\mu-\gamma\big]-\beta\,\rho\,e^{[\theta+(1-\eta)\lambda_1]\xi}\Big\}\\
&\!\!\ge\!\!&\beta\,\rho\,e^{\eta\lambda_1\xi}\big(1-e^{(\theta+(1-\eta)\lambda_1)\xi}\big)\ge0
\eeaa
by the choices of $\eta$ and $q$. Therefore, the proof of this lemma has been completed.
\end{proof}


\subsection{An auxiliary truncated problem}

Now, given $l>-\xi_2\,(>0)$, we consider the following truncated problem
\begin{eqnarray}\label{eq:TW-a}
\left\{\begin{array}{ll}
D[\phi]-c\,\phi'+\mu\,(1-\phi)-\beta\,\phi\,\psi=0 & \mbox{ in }[-l, l],\vspace{3pt}\\
d\,D[\psi]-c\,\psi'-(\mu+\gamma)\,\psi+\beta\,\phi\,\psi=0 & \mbox{ in }[-l, l],\vspace{3pt}\\
(\phi, \psi)=(\overline{\phi},\overline{\psi}) & \mbox{ on }(-\infty,-l),\vspace{3pt}\\
(\phi, \psi)=(\phi(l),\psi(l)) & \mbox{ on }(l,+\infty),\end{array}\right.
\end{eqnarray}
where
$$\left\{\begin{array}{ll}
\phi'(-l):=\displaystyle\lim_{h\searrow0}\frac{\phi(-l+h)-\phi(-l)}{h}, & \displaystyle\psi'(-l):=\lim_{h\searrow0}\frac{\psi(-l+h)-\psi(-l)}{h},\vspace{3pt}\\
\phi'(l):=\displaystyle\lim_{h\searrow0}\frac{\phi(l)-\phi(l-h)}{h}, & \displaystyle\psi'(l):=\lim_{h\searrow0}\frac{\psi(l)-\psi(l-h)}{h}.\end{array}\right.$$

Next, we give some notations. Set $\mathcal{C}^l:=C([-l, l])\times C([-l, l])$ and
$$\mathcal{S}^l:=\big\{(\phi, \psi)\in \mathcal{C}^l\ |\ \underline{\phi}\le\phi\le\overline{\phi},\; \underline{\psi}\le\psi\le\overline{\psi}\mbox{ in }[-l, l]\hbox{ and }(\phi,\psi)(-l)=(\overline{\phi},\overline{\psi})(-l)\big\}.$$
From the definition of $\overline{\phi}, \underline{\phi}, \overline{\psi}, \underline{\psi}$, we know that $0\le\phi\le1$ and $0\le\psi\le e^{\lambda_1l}$ in $[-l, l]$ for any $(\phi, \psi)\in\mathcal{S}^l$. Hence $\mathcal{S}^l$ is a nonempty bounded closed convex set in $(\mathcal{C}^l, \|\cdot\|)$, where $\|\cdot\|$ is the usual sup norm. For any $(\phi,\psi)\in\mathcal{S}^l$, we extend $(\phi,\psi)$ be continuity outside the interval $[-l,l]$ as in~\eqref{eq:TW-a} and we introduce the continuous functions $H_1^l(\phi,\psi)$ and $H_2^l(\phi,\psi)$ defined in $\R$ by
$$\left\{\begin{array}{rcl}
H_1^l(\phi, \psi)(\xi) & = & \alpha\,\phi(\xi)+D[\phi](\xi)+\mu\,(1-\phi(\xi))-\beta\,\phi(\xi)\,\psi(\xi),\vspace{3pt}\\
H_2^l(\phi, \psi)(\xi) & = & \alpha\,\psi(\xi)+d\,D[\psi](\xi)-(\mu+\gamma)\,\psi(\xi)+\beta\,\phi(\xi)\,\psi(\xi),\end{array}\right.$$
where $\alpha=\alpha^l$ is a positive constant such that
$$\alpha>\max\big\{2+\mu+\beta\,e^{\lambda_1l}, 2\,d+\mu+\gamma\big\}.$$
For $(\phi_i, \psi_i)\in \mathcal{S}^l$, $i=1,2$, with $\phi_1\le\phi _2$ and $\psi_1\le\psi_2$ in $[-l, l]$, we have
\be\label{ineq:H}
H_1^l(\phi_1, \psi_2)(\xi)\le H_1^l(\phi_1, \psi_1)(\xi)\le H_1^l(\phi_2, \psi_1)(\xi)\ \hbox{ and }\  H_2^l(\phi_1, \psi_1)(\xi)\le H_2^l(\phi_2, \psi_2)(\xi)
\ee
for all $\xi\in[-l, l]$. Finally, we define the operator $F^l=(F_1^l, F_2^l)$ from $\mathcal{S}^l$ into $\mathcal{C}^l$ as follows
$$\left\{\begin{array}{rcl}
F_1^l(\phi, \psi)(\xi)&=&\displaystyle e^{\alpha(-l-\xi)/c}\,\overline{\phi}(-l)+\int_{-l}^\xi\frac{e^{\alpha(z-\xi)/c}}{c}H_1^l(\phi, \psi)(z)\,dz,\ \ \xi\in[-l,l],\vspace{3pt}\\
F_2^l(\phi, \psi)(\xi)&=&\displaystyle e^{\alpha(-l-\xi)/c}\,\overline{\psi}(-l)+\int_{-l}^\xi\frac{e^{\alpha(z-\xi)/c}}{c}H_2^l(\phi, \psi)(z)\,dz,\ \ \xi\in[-l,l].\end{array}\right.$$
Note that a fixed point $(\phi, \psi)$ of the operator $F^l$, extended outside the interval $[-l,l]$ as in~\eqref{eq:TW-a}, gives a solution of~\eqref{eq:TW-a} which is continuous in $\R$ and of class $C^1(\R\backslash\{-l,l\})$.

To show the existence of such a fixed point, we apply Schauder's fixed point theorem in the next lemma.

\begin{lemma}\label{lem:exist-l}
Given $l>-\xi_2$, there exists a $C(\R)\times C(\R)$ and $C^1(\R\backslash\{-l,l\})\times C^1(\R\backslash\{-l,l\})$ solution $(\phi, \psi)$ of~\eqref{eq:TW-a} such that
\be\label{ineq:phi-psi-l}
0\le\underline{\phi}\le\phi\le1\quad \mbox{and}\quad 0\le\underline{\psi}\le\psi\le\overline{\psi}\quad\hbox{in }(-\infty, l].
\ee
\end{lemma}

\begin{proof}
First, we claim that $F^l(\mathcal{S}^l)\subset\mathcal{S}^l$. By~\eqref{ineq:H}, for any $(\phi, \psi)\in\mathcal{S}^l$, we have
$$F_1^l(\underline{\phi}, \overline{\psi})\le F_1^l(\phi, \psi)\le F_1^l(\overline{\phi}, \underline{\psi})\ \hbox{ and }\ F_2^l(\underline{\phi}, \underline{\psi})\le F_2^l(\phi, \psi)\le F_1^l(\overline{\phi}, \overline{\psi})\ \mbox{ in $[-l, l]$.}$$
By Lemma~\ref{lem:UL} and the definition of the upper and lower solutions, we also derive that
$$\underline{\phi}\le F_1^l(\underline{\phi}, \overline{\psi}),\ \ F_1^l(\overline{\phi}, \underline{\psi})\le\overline{\phi},\ \ \underline{\psi}\le F_2^l(\underline{\phi}, \underline{\psi})\ \hbox{ and }\ F_2^l(\overline{\phi}, \overline{\psi})\le\overline{\psi}\ \hbox{ in }[-l,l].$$
Hence $F^l(\mathcal{S}^l)\subset\mathcal{S}^l$.

By using Arzela-Ascoli Theorem, the operator $F^l:\mathcal{S}^l\rightarrow\mathcal{S}^l$ is completely continuous with respect to the sup norm. With the help of Schauder's fixed point theorem, we conclude that there exists a pair $(\phi, \psi)\in\mathcal{S}^l$ such that $(\phi, \psi)=F^l(\phi, \psi)$. Therefore, $(\phi, \psi)$, extended outside the interval $[-l,l]$ as in~\eqref{eq:TW-a}, solves~\eqref{eq:TW-a} and satisfies the properties stated in Lemma~\ref{lem:exist-l}.
\end{proof}


\section{Existence of a traveling wave for $c>c^*$}\label{sec3}
\setcounter{equation}{0}


\subsection{Proof of Theorem~\ref{th:exist} for $c>c^*$}\label{sec31}

In this section, we show Theorem~\ref{th:exist} for any fixed real number $c\in(c^*,+\infty)$. Namely, we show the existence of a bounded solution $(\phi, \psi)$ of~\eqref{eq:TW} satisfying $0<\phi<1$ in $\BbR$, $\psi>0$ in $\BbR$, and such that~\eqref{bc-minus} and~\eqref{bc-plus} hold.

First, we consider a positive increasing sequence $\{l_k\}_{k\in\mathbb{N}}$ such that $l_k\rightarrow\infty$ as $k\rightarrow\infty$, and $l_k>-\xi_2$ for all $k\in\mathbb{N}$, where $\xi_2<0$ is as in~\eqref{xi12}. By Lemma~\ref{lem:exist-l}, for each $k\in\mathbb{N}$, there exists a $C(\R)\cap C^1(\R\backslash\{-l_k,l_k\})$ solution $(\phi_k, \psi_k)$ of~\eqref{eq:TW-a} and~\eqref{ineq:phi-psi-l} for $l=l_k$. For each $K\in\mathbb{N}$ such that $l_K\ge2$, since $\overline{\psi}$ is bounded above in $[-l_K, l_K]$, it follows from~\eqref{ineq:phi-psi-l} that the sequences
$$\{\phi_k\}_{k\ge K},\quad \{\psi_k\}_{k\ge K},\quad \{\phi_k\psi_k\}_{k\ge K}$$
are uniformly bounded on $[-l_K, l_K]$. Also, the sequences $\{\phi_k'\}_{k\ge K}$ and $\{\psi_k'\}_{k\ge K}$ are uniformly bounded in $[-l_K+1, l_K-1]$, due to~\eqref{eq:TW-a} and~\eqref{ineq:phi-psi-l}. Since $\phi_k''(\xi)$ and $\psi_k''(\xi)$ can be expressed in terms of $\phi_k(\xi)$, $\psi_k(\xi)$, $\phi_k(\xi\pm1)$, $\psi_k(\xi\pm1)$, $\phi_k(\xi\pm2)$, $\psi_k(\xi\pm2)$, $\phi_k'(\xi)$ and $\psi_k'(\xi)$ in $[-l_K+2,l_K-2]$, one infers that the sequences $\{\phi_k''\}_{k\ge K}$ and $\{\psi_k''\}_{k\ge K}$ are uniformly bounded in $[-l_K+2, l_K-2]$. By using Arzela-Ascoli theorem on $[-l_K+2,l_K-2]$ for every $K\in\mathbb{N}$ large enough, we obtain a subsequence $\{(\phi_{k_j},\psi_{k_j})\}$ of $\{(\phi_{k},\psi_{k})\}$ through the diagonal process such that
$$\phi_{k_j}\rightarrow\phi,\ \psi_{k_j}\rightarrow\psi,\ \phi_{k_j}'\rightarrow\phi',\ \psi_{k_j}'\rightarrow\psi'\ \hbox{ as }j\to+\infty$$
uniformly in any compact subinterval of $\BbR$, for some functions $\phi\in C^1(\BbR)$ and $\psi\in C^1(\BbR)$. Then $(\phi, \psi)$ is a solution of the system~\eqref{eq:TW} with
\be\label{ineqphipsi}
0\le\underline{\phi}\le\phi\le1\ \hbox{ and }\ 0\le\underline{\psi}\le\psi\le\overline{\psi}\ \hbox{ in }\BbR.
\ee
By the definitions of $\underline{\phi}$, $\overline{\psi}$ and $\underline{\psi}$, it easy to check that
$$(\phi, \psi)(-\infty)=(1,0).$$
Notice also that, by differentiating the equations~\eqref{eq:TW}, one infers by induction that the functions $\phi$ and $\psi$ are of class $C^{\infty}$ in $\R$.

\begin{lemma}\label{lemnontrivial0}
The functions $\phi$ and $\psi$ are non-trivial, in the sense that
$$0<\phi<1\ \hbox{ and }\ \psi>0\ \hbox{ in }\BbR.$$
\end{lemma}

\begin{proof}
Firstly, owing to the definition of $\underline{\psi}$, we have $\psi>0$ in $(-\infty, \xi_2)$. For contradiction, we assume that there exists a real number $\xi_0\in[\xi_2,+\infty)$ such that $\psi(\xi_0)=0$ and $\psi(\xi)>0$ for all $\xi<\xi_0$. Since $\psi\ge0$ in $\BbR$, we also have $\psi'(\xi_0)=0$. From the second equation of~\eqref{eq:TW}, we get that $\psi(\xi_0-1)=\psi(\xi_0+1)=0$, a contradiction to the definition of $\xi_0$.

Let us now show that $\phi>0$ over $\BbR$. Indeed, if $\phi(\xi^*)=0$ for some real number $\xi^*$, then
$$0=-c\,\phi'(\xi^*)+D[\phi](\xi^*)+\mu\,\big(1-\phi(\xi^*)\big)-\beta\,\phi(\xi^*)\,\psi(\xi^*)=-c\,\phi'(\xi^*)+D[\phi](\xi^*)+\mu>0,$$
since $\phi'(\xi^*)=0$, $D[\phi](\xi^*)\ge0$ and $\mu>0$. This contradiction leads to the inequality $\phi>0$ in $\BbR$.

Similarly, we claim that $\phi<1$ in $\BbR$ by a contradiction argument. If there exists a real number $\tilde{\xi}$ such that $\phi(\tilde{\xi})=1$, then
$$0=-c\,\phi'(\tilde{\xi})+D[\phi](\tilde{\xi})+\mu\,\big(1-\phi(\tilde{\xi})\big)-\beta\,\phi(\tilde{\xi})\,\psi(\tilde{\xi})=-c\,\phi'(\tilde{\xi})+D[\phi](\tilde{\xi})-\beta\,\psi(\tilde{\xi})<0,$$
since $\phi'(\tilde{\xi})=0$, $D[\phi](\tilde{\xi})\le0$ and $\psi(\tilde{\xi})>0$. This contradiction leads to the inequality $\phi<1$ in $\BbR$.
\end{proof}

The next main step consists in showing that the function $\psi$ is actually bounded. A first key-point is the following Harnack type property for equations of the type~\eqref{eq:TW} satisfied by the second component~$\psi$. We state this property in a more general framework.

\begin{lemma}\label{lemharnack}
Let $M$ be a positive real number. Then there exists a constant $C=C(M)>0$ such that, for any continuous functions $a$ and $b$ with $M^{-1}\le a(\xi)\le M$ and $b(\xi)\ge-M$ for all $\xi\in\R$ and for any positive $C^1(\R)$ function $u$ satisfying
$$u'(\xi)\ge a(\xi)\,u(\xi+1)+b(\xi)\,u(\xi)\ \hbox{ for all }\xi\in\R,$$
there holds
$$C^{-1}\le\frac{u(\xi+1)}{u(\xi)}\le C\ \hbox{ for all }\xi\in\R.$$
\end{lemma}

In order not to lengthen too much the main line of the proof of Theorem~\ref{th:exist} with $c>c^*$, the proof of Lemma~\ref{lemharnack} is postponed in Section~\ref{sec32}.

Coming back to our solutions $(\phi,\psi)$ of~\eqref{eq:TW}, since $c>0$ and $\phi$ is nonnegative, it follows from Lemma~\ref{lemharnack} applied to the positive function $u=\psi$ solving $\psi'(\xi)\ge(d/c)\psi(\xi+1)-(2d/c+\mu/c+\gamma/c)\,\psi(\xi)$ that the functions $\xi\mapsto\psi(\xi\pm1)/\psi(\xi)$ are bounded in $\R$. Hence, from the equation~\eqref{eq:TW} itself and since $\phi$ is bounded, the function
$$\xi\mapsto\frac{\psi'(\xi)}{\psi(\xi)}$$
is therefore bounded too.

The following two lemmas deal with the behavior of $\phi$ and $\psi$ at $+\infty$ if $\limsup_{\xi\to+\infty}\psi(\xi)=+\infty$. The first one says that $\phi$ is small when $\psi$ is large. This property actually holds locally uniformly with respect to the speed $c$. It is stated in this more general framework since it will be used again in Section~\ref{secc*} to get the existence of a bounded solution~$(\phi,\psi)$ of~\eqref{eq:TW} with speed $c^*$.

\begin{lemma}\label{leminfty}
Let $0<\underline{c}\le\overline{c}$ be two given positive real numbers. Let $\{c_k\}$ be a sequence of real numbers in $[\underline{c},\overline{c}]$ and let $\{(\phi_k,\psi_k)\}$ be a sequence of solutions of~\eqref{eq:TW} with speed $c_k$ and satisfying~\eqref{nontrivial}. If $\{\xi_k\}$ is a sequence of real numbers such that $\psi_k(\xi_k)\to+\infty$ as $k\to+\infty$, then $\phi_k(\xi_k)\to0$ as $k\to+\infty$.
\end{lemma}

Since this lemma is concerned with general sequences of solutions with different speeds, and in order not to lengthen too much the main line of the proof of Theorem~\ref{th:exist} with given speed $c>c^*$, the proof of Lemma~\ref{leminfty} is postponed in Section~\ref{sec32}.

Coming back to our solution $(\phi,\psi)$ of~\eqref{eq:TW} satisfying~\eqref{nontrivial} and~\eqref{bc-minus}, the following result shows the convergence of $\psi$ to $+\infty$ at $+\infty$ if it were not bounded.

\begin{lemma}\label{leminfty2}
If $\limsup_{\xi\to+\infty}\psi(\xi)=+\infty$, then $\lim_{\xi\to+\infty}\psi(\xi)=+\infty$.
\end{lemma}

\begin{proof}
Assume by way of contradiction that $\limsup_{\xi\to+\infty}\psi(\xi)=+\infty$ and $\liminf_{\xi\to+\infty}\psi(\xi)<+\infty$. Since $\psi'/\psi$ is globally bounded, there are then $M\in\R$ and two sequences $\{\theta_k\}$ and $\{\xi_k\}$ converging to $+\infty$ and such that
$$\psi(\theta_k)\le M,\ \ \theta_k<\xi_k-1<\xi_k<\xi_k+1<\theta_{k+1},\ \ \psi(\xi_k)=\max_{[\theta_k,\theta_{k+1}]}\psi\ \ \Big(\!=\max_{[\xi_k-1,\xi_k+1]}\psi\Big)$$
for all $k\in\N$ and $\lim_{k\to+\infty}\psi(\xi_k)=+\infty$. Therefore, $\psi'(\xi_k)=0$ and $d\,D[\psi](\xi_k)\le0$. Hence, by~\eqref{eq:TW}, one infers that $\big(\mu+\gamma-\beta\,\phi(\xi_k)\big)\,\psi(\xi_k)\le0$ for all $k\in\N$. This is clearly impossible for large $k$ since $\psi(\xi_k)>0$, and $\phi(\xi_k)\to0$ as $k\to+\infty$ by Lemma~\ref{leminfty}. The proof is thereby complete.
\end{proof}

To proceed further, we recall the following useful fundamental theory from~\cite{CG2} (or~\cite{CFG}) in dealing with the asymptotic tail behavior of wave profiles for a lattice dynamical system.

\begin{proposition}\label{key}{\rm{\cite{CG2}}}
Let $\varsigma>0$ be a positive constant, let $B:\R\to\R$ be a continuous function having finite $B(\pm\infty):=\lim_{x\rightarrow\pm\infty}B(x)$ and let $z$ be a continuous function such that
\be\label{eqz}
\varsigma\,z(x)=e^{\int_{x}^{x+1}z(s)ds}+e^{\int_{x}^{x-1}z(s)ds}+B(x),\ \forall x\in\mathbb{R}.
\ee
Then $z$ is uniformly continuous and bounded in $\BbR$. In addition, the limits $\omega^{\pm}=\lim_{x\rightarrow\pm\infty} z(x)$ exist and are real roots of the characteristic equations
$$\varsigma\,\omega=e^{\omega}+ e^{-\omega}+B(\pm\infty).$$
\end{proposition}

With this result and the previous lemmas in hand, we can show that $\psi$ is bounded in $\R$.

\begin{lemma}\label{lempsi}
The function $\psi$ is bounded.
\end{lemma}

\begin{proof}
Assume not. Then $\limsup_{\xi\to +\infty}\psi(\xi)=+\infty$, since $\psi$ is continuous, positive, and $\psi(-\infty)=0$. Therefore, Lemmas~\ref{leminfty} and~\ref{leminfty2} imply that $\psi(\xi)\to+\infty$ and $\phi(\xi)\to0$ as $\xi\to+\infty$. From~\eqref{eq:TW}, the continuous function $z:=\psi'/\psi$ satisfies
$$\frac{c}{d}\,z(x)=e^{\int_{x}^{x+1}z(s)ds}+e^{\int_{x}^{x-1}z(s)ds}-2-\frac{\mu+\gamma}{d}+\frac{\beta\,\phi(x)}{d}$$
for all $x\in\R$. Since $\phi$ has finite limits at $\pm\infty$ and $\phi(+\infty)=0$, it then follows from Proposition~\ref{key} that, in particular, $z$ has a finite limit $\omega$ at $+\infty$, with
\be\label{eqomega}
d\,\big(e^{\omega}+e^{-\omega}-2\big)=c\,\omega+\mu+\gamma.
\ee
Since $\mu$ and $\gamma$ are positive, this equation has a negative and a positive root. The function $z=\psi'/\psi$ cannot converge to the negative root at $+\infty$, since $\psi(+\infty)=+\infty$. Therefore, $\psi'/\psi$ converges at $+\infty$ to the positive root $\omega$ of~\eqref{eqomega}. Remember now that $\lambda_1<\lambda_2$ are the two positive roots of equation~\eqref{psi-char1}. Since $\beta>0$, one infers immediately that $\lambda_1<\lambda_2<\omega$. But $\lim_{\xi\to+\infty}\psi'(\xi)/\psi(\xi)=\omega>0$ yields $\ln\psi(\xi)\sim\omega\,\xi$ as $\xi\to+\infty$, while~\eqref{ineqphipsi} implies that $\psi(\xi)\le\overline{\psi}(\xi)=e^{\lambda_1\xi}$ for all $\xi\in\R$. One gets a contradiction, since $\lambda_1<\omega$. As a conclusion, the function $\psi$ is bounded and the proof of Lemma~\ref{lempsi} is complete.
\end{proof}

To complete the proof of Theorem~\ref{th:exist} in case $c>c^*$, we show in the following lemmas that none of the components $\phi$ and $\psi$ can be trivial at $+\infty$.

\begin{lemma}\label{leminfphi}
There holds $\inf_{\R}\phi>0$.
\end{lemma}

\begin{proof}
Remember that the $C^{\infty}$ function $\phi$ satisfies $0<\phi<1$ in $\R$ and $\phi(-\infty)=1$. Assume by contradiction that $\inf_{\R}\phi=0$. Then there exists a sequence $\{\xi_k\}$ converging to $+\infty$ such that $\phi(\xi_k)\to0$ as $k\to+\infty$. On the other hand, since both functions $\phi$ and $\psi$ are bounded, the equations~\eqref{eq:TW} guarantee that the functions $\phi$ and $\psi$ have bounded derivatives at any order. Therefore, by Arzela-Ascoli theorem, the functions $\xi\mapsto\phi(\xi+\xi_k)$ and $\xi\mapsto\psi(\xi+\xi_k)$ converge in $C^{\infty}_{loc}(\R)$ as $k\to+\infty$, up to extraction of a subsequence, to some nonnegative $C^{\infty}$ functions $\phi_{\infty}$ and $\psi_{\infty}$. Furthermore,
\be\label{eqphiinfty}
-c\,\phi_{\infty}'+D[\phi_{\infty}]+\mu\,(1-\phi_{\infty})-\beta\,\phi_{\infty}\,\psi_{\infty}=0
\ee
in $\R$ and $\phi_{\infty}(0)=0$. Since $0$ is a global minimum of $\phi_{\infty}$, one has $\phi_{\infty}'(0)=0$ and the above equality at $0$ leads to a contradiction, since $\phi_{\infty}\ge0$ and $\mu>0$. Therefore, $\inf_{\R}\phi>0$.
\end{proof}

To show that $\psi$ cannot approach $0$ at $+\infty$, even for a sequence, the key-step is the following lemma saying that $\psi$ is increasing when it is small. The property actually holds locally uniformly with respect to the speed $c$ and we state the lemma in this slightly more general framework, since it will be used as such in Section~\ref{secc*}.

\begin{lemma}\label{lempsi'}
Let $0<\underline{c}\le\overline{c}$ be two given positive real numbers. There is $\epsilon>0$ such that, for any $\Gamma\in[\underline{c},\overline{c}]$ and for any solution $(\Phi,\Psi)$ of~\eqref{eq:TW} $($with speed $\Gamma$ in place of $c$$)$ satisfying~\eqref{nontrivial}, there holds
$$\forall\,\xi\in\R,\quad\big(\Psi(\xi)\le\epsilon\big)\Longrightarrow\big(\Psi'(\xi)>0).$$
\end{lemma}

In order to conclude now the proof of Theorem~\ref{th:exist} with $c>c^*$, the proof of Lemma~\ref{lempsi'} is postponed in Section~\ref{sec32}. Coming back to our solution $(\phi,\psi)$, we immediately get from Lemma~\ref{lempsi'} and the positivity of $\psi$ in $\R$ that
\be\label{liminfpsi}
\liminf_{\xi\to+\infty}\psi(\xi)>0.
\ee

We also claim that
\be\label{claimphi}
\limsup_{\xi\to+\infty}\phi(\xi)<1.
\ee
Indeed, otherwise, there exists a sequence of real numbers $\{\xi_k\}$ converging to $+\infty$ such that $\phi(\xi_k)\to1$ as $k\to+\infty$. As in the proof of Lemma~\ref{leminfphi}, up to extraction of a subsequence, the functions $\xi\mapsto\phi(\xi+\xi_k)$ and $\xi\mapsto\psi(\xi+\xi_k)$ converge as $k\to+\infty$ in $C^{\infty}_{loc}(\R)$ to some nonnegative $C^{\infty}$ functions $\phi_{\infty}$ and $\psi_{\infty}$ solving~\eqref{eq:TW}. Furthermore, $0<\phi_{\infty}\le1$ and $\psi_{\infty}>0$ in $\R$ from Lemma~\ref{leminfphi} and~\eqref{liminfpsi}. Since $\phi_{\infty}(0)=1$, one has $\phi'_{\infty}(0)=0$. The equation~\eqref{eqphiinfty} satisfied by~$\phi_{\infty}$ at $0$ leads to a contradiction, since $D[\phi_{\infty}](0)\le0$ and $-\beta\,\phi_{\infty}(0)\,\psi_{\infty}(0)=-\beta\,\psi_{\infty}(0)<0$. Therefore, the claim~\eqref{claimphi} holds.

In order to complete the proof of~\eqref{bc-plus}, let us finally show that
\be\label{liminfsup}
\liminf_{\xi\to+\infty}\phi(\xi)\le s^*\le\limsup_{\xi\to+\infty}\phi(\xi)\ \hbox{ and }\ \liminf_{\xi\to+\infty}\psi(\xi)\le e^*\le\limsup_{\xi\to+\infty}\psi(\xi).
\ee
Call $\phi_-=\liminf_{\xi\to+\infty}\phi(\xi)$, $\phi_+=\limsup_{\xi\to+\infty}\phi(\xi)$, $\psi_-=\liminf_{\xi\to+\infty}\psi(\xi)$ and $\psi_+=\limsup_{\xi\to+\infty}\psi(\xi)$. One already knows from~\eqref{liminfpsi},~\eqref{claimphi} and Lemmas~\ref{lempsi} and~\ref{leminfphi} that
$$0<\phi_-\le\phi_+<1\ \hbox{ and }\ 0<\psi_-\le\psi_+<+\infty.$$
Consider now a sequence $\{\xi_k\}$ converging to $+\infty$ such that $\psi(\xi_k)\to\psi_-$ as $k\to+\infty$. Up to extraction of a subsequence (as for instance in the proof of Lemma~\ref{leminfphi}), the functions $\xi\mapsto\phi(\xi+\xi_k)$ and $\xi\mapsto\psi(\xi+\xi_k)$ converge in $C^{\infty}_{loc}(\R)$ to some bounded functions $0<\phi_{\infty}<1$ and $\psi_{\infty}>0$ satisfying~\eqref{eq:TW}. Furthermore, $0<\psi_-=\psi_{\infty}(0)=\min_{\R}\psi_{\infty}$. Therefore, $\psi_{\infty}'(0)=0$ and $D[\psi_{\infty}](0)\ge0$. Hence
$$-(\mu+\gamma)\,\psi_-+\beta\,\phi_{\infty}(0)\,\psi_-\le0,$$
that is, $\beta\,\phi_{\infty}(0)\le\mu+\gamma$. This yields $\phi_-=\liminf_{\xi\to+\infty}\phi(\xi)\le(\mu+\gamma)/\beta=1/\sigma=s^*$. Similarly, it follows that $\phi_+=\limsup_{\xi\to+\infty}\phi(\xi)\ge s^*$. Consider also a sequence $\{\zeta_k\}$ converging to $+\infty$ such that $\phi(\xi_k)\to\phi_-$ as $k\to+\infty$. As above, up to extraction of a subsequence, the functions $\xi\mapsto\phi(\xi+\zeta_k)$ and $\xi\mapsto\psi(\xi+\zeta_k)$ converge in $C^{\infty}_{loc}(\R)$ to some bounded functions $0<\Phi_{\infty}<1$ and $\Psi_{\infty}>0$ satisfying~\eqref{eq:TW}. Furthermore, $0<\phi_-=\Phi_{\infty}(0)=\min_{\R}\Phi_{\infty}$. Therefore, $\Phi_{\infty}'(0)=0$ and $D[\Phi_{\infty}](0)\ge0$. Hence
$$\mu\,(1-\phi_-)-\beta\,\phi_-\,\Psi_{\infty}(0)\le0.$$
Since $0<\phi_-\le s^*=1/\sigma$, one gets immediately that $\Psi_{\infty}(0)\ge(\mu/\beta)(\sigma-1)=e^*$, whence $\psi_+=\limsup_{\xi\to+\infty}\psi(\xi)\ge e^*$. Similarly, it follows that $\psi_-=\liminf_{\xi\to+\infty}\psi(\xi)\le e^*$.

As a conclusion, \eqref{bc-plus} is proved and the proof of Theorem~\ref{th:exist} in case $c>c^*$ is thereby complete.\hfill\break\par

As explained after the statement of Theorem~\ref{th:exist} in Section~\ref{intro}, the question of the existence of a limit of $(\phi,\psi)$ at $+\infty$ is unclear. However, we can say that the a priori existence of a limit of one of these two functions guarantees the convergence of both, and that the endemic state $(s^*,e^*)$ defined in~\eqref{endemic} is the only possible limit.

\begin{lemma}\label{lemendemic}
Let $(\phi,\psi)$ be a bounded classical solution of~\eqref{eq:TW} satisfying~\eqref{nontrivial},~\eqref{bc-minus} and~\eqref{bc-plus}, with speed $c\ge c^*$. If $\phi(+\infty)$ or $\psi(+\infty)$ exists, then they both exist and
$$(\phi(+\infty),\psi(+\infty))=(s^*,e^*).$$
\end{lemma}

\begin{proof}
Assume first that $l=\lim_{\xi\to+\infty}\phi(\xi)$ exists. Property~\eqref{bc-plus} yields $0<l=s^*<1$. Consider now any sequence $\{\xi_k\}$ converging to $+\infty$. Up to extraction of a subsequence, the functions $\xi\mapsto\phi(\xi+\xi_k)$ and $\xi\mapsto\psi(\xi+\xi_k)$ converge in $C^{\infty}_{loc}(\R)$ to some functions $\phi_{\infty}=l=s^*$ and $\psi_{\infty}$ such that
$$\mu\,(1-s^*)-\beta\,s^*\,\psi_{\infty}(\xi)=0\ \hbox{ for all }\xi\in\R.$$
Therefore, the function $\psi_{\infty}$ is identically equal to the constant $\mu(1-s^*)/(\beta s^*)=e^*$. Since the limit does not depend on the sequence $\{\xi_k\}$, one infers that $\lim_{\xi\to+\infty}\psi(\xi)=e^*$.

Conversely, if $L=\lim_{\xi\to+\infty}\psi(\xi)$ exists, property~\eqref{bc-plus} yields $0<L=e^*$. For any sequence $\{\xi_k\}$ converging to $+\infty$, the functions $\xi\mapsto\phi(\xi+\xi_k)$ and $\xi\mapsto\psi(\xi+\xi_k)$ converge in $C^{\infty}_{loc}(\R)$, up to a subsequence, to some functions $\phi_{\infty}$ and $\psi_{\infty}=L=e^*$ such that
$$-(\mu+\gamma)\,e^*+\beta\,e^*\,\phi_{\infty}(\xi)=0\ \hbox{ for all }\xi\in\R.$$
Therefore, the function $\phi_{\infty}$ is identically equal to the constant $(\mu+\gamma)/\beta=s^*$. Since the limit does not depend on the sequence $\{\xi_k\}$, one infers that  $\lim_{\xi\to+\infty}\phi(\xi)=s^*$.

Therefore, if the limit $l=\phi(+\infty)$ or the limit $L=\psi(+\infty)$ exists, then they both exist such that $(\phi(+\infty),\psi(+\infty))=(s^*,e^*)$.
\end{proof}

\begin{rem}{\rm
The condition $c\ge c^*$ in Lemma~\ref{lemendemic} is not a restriction, since we shall prove in Section~\ref{sec5} that, for any solution $(\phi,\psi)$ of~\eqref{eq:TW} satisfying~\eqref{nontrivial} and~\eqref{bc-minus} with speed $c$, there holds $c\ge c^*$.}
\end{rem}


\subsection{Proof of Lemmas~\ref{lemharnack},~\ref{leminfty} and~\ref{lempsi'}}\label{sec32}

In this section, we prove some technical lemmas stated in Section~\ref{sec31}.\hfill\break

\noindent{\it{Proof of Lemma~$\ref{lemharnack}$.}} Although the idea of the proof is similar to the one given in \cite{CG2}, we provide the details here for completeness. Up to multiplication of $u$ by a positive constant and up to a shift in space, one can assume without loss of generality that $u(0)=1$ and it is sufficient to show that $u(\pm 1)\le C=C(M)$. Firstly, since $u'(\xi)\ge-M\,u(\xi)$ for all $\xi\in\R$, the function $\xi\mapsto v(\xi):=u(\xi)\,e^{M\xi}$ is nondecreasing, hence
$$u(-1)\le e^M\,u(0)=e^M.$$
Secondly, for all $\xi\in[0,1]$, one has
$$v'(\xi)=(u'(\xi)+Mu(\xi))\,e^{M\xi}\ge a(\xi)\,u(\xi+1)\,e^{M\xi}\ge\frac{v(\xi+1)\,e^{-M}}{M}\ge\frac{v(1)\,e^{-M}}{M}=\frac{u(1)}{M}.$$
Hence, $v(\xi)\ge v(0)+u(1)\,\xi/M=1+u(1)\,\xi/M$ for all $\xi\in[0,1]$. In other words,
$$u(\xi)\ge\Big(1+\frac{u(1)\,\xi}{M}\Big)\,e^{-M\xi}\ \hbox{ for all }\xi\in[0,1].$$
Finally, for all $\xi\in[-1/2,0]$,
$$v'(\xi)\ge a(\xi)u(\xi+1)\,e^{M\xi}\ge\frac{e^{M\xi}}{M}\times\Big(1+\frac{u(1)\,(\xi+1)}{M}\Big)\,e^{-M(\xi+1)}\ge\frac{e^{-M}}{M}\times\Big(1+\frac{u(1)}{2M}\Big).$$
Therefore,
$$1=v(0)\ge\underbrace{v(-1/2)}_{\ge0}+\frac{e^{-M}}{2M}\times\Big(1+\frac{u(1)}{2M}\Big)\ge\frac{e^{-M}}{2M}\times\Big(1+\frac{u(1)}{2M}\Big).$$
Hence
$$u(1)\le 2\,M\,\big(2\,M\,e^M-1\big)$$
and the proof of Lemma~\ref{lemharnack} is thereby complete with $C(M)=\max\{e^M,2M(2Me^M-1)\}$.\hfill$\Box$\break

\noindent{\it{Proof of Lemma~$\ref{leminfty}$.}} Let $0<\underline{c}\le\overline{c}$, $\{c_k\}$, $\{(\phi_k,\psi_k)\}$ and $\{\xi_k\}$ be as in the statement and assume by way of contradiction that there are $\epsilon>0$ and a subsequence, still denoted with the same index $k$, such that $\psi_k(\xi_k)\to+\infty$ as $k\to+\infty$ and $\phi_k(\xi_k)\ge\epsilon$ for all $k\in\N$. Since $0<\phi_k<1$ and $\psi_k>0$ in $\R$, the equation~\eqref{eq:TW} for $\phi_k$ (with $c_k\in[\underline{c},\overline{c}]$) implies that $\phi_k'\le(2+\mu)/\underline{c}$ in $\R$. Hence
\be\label{phik}
\phi_k(\xi)\ge\frac{\epsilon}{2}\ \hbox{ for all }\xi\in[\xi_k-\delta,\xi_k]\hbox{ and for all }k\in\N,
\ee
where $\delta=\epsilon\,\underline{c}/(4+2\mu)>0$. On the other hand, since
$$\psi_k'(\xi)\ge\frac{d}{\overline{c}}\,\psi_k(\xi+1)-\frac{2d+\mu+\gamma}{\underline{c}}\,\psi_k(\xi)\ \hbox{ for all }\xi\in\R\hbox{ and for all }k\in\N,$$
Lemma~\ref{lemharnack} applied to the positive functions $\psi_k$ implies that the functions $\xi\mapsto\psi_k(\xi\pm1)/\psi_k(\xi)$ are globally bounded independently of $k\in\mathbb{N}$. Hence, the functions $\psi_k'/\psi_k$ are globally bounded in $\R$ independently of $k\in\mathbb{N}$. Therefore, the limit $\lim_{k\to+\infty}\psi_k(\xi_k)=+\infty$ implies that $0<M_k:=\min_{[\xi_k-\delta,\xi_k]}\psi_k\to+\infty$ as $k\to+\infty$. Now, equation~\eqref{eq:TW} and the inequalities $0<\phi_k<1$ and~\eqref{phik} yield
$$\max_{[\xi_k-\delta,\xi_k]}\phi_k'\le\frac{2+\mu}{\underline{c}}-\frac{\beta\,\epsilon\,M_k}{2\,\overline{c}}\to-\infty\ \hbox{ as }k\to+\infty.$$
This contradicts the global boundedness of the functions $\phi_k$. The proof of Lemma~\ref{leminfty} is thereby complete.\hfill$\Box$\break

\noindent{\it{Proof of Lemma~$\ref{lempsi'}$.}} Assume by way of contradiction that there is no such $\epsilon$. Then there exist a sequence of real numbers $\{c_k\}$ in $[\underline{c},\overline{c}]$, a sequence of solutions $\{(\phi_k,\psi_k)\}$ of~\eqref{eq:TW} with speed $c=c_k$ and $0<\phi_k<1$, $\psi_k>0$ in $\R$, and a sequence of real numbers $\{\xi_k\}$ such that
\be\label{psik}
\psi_k(\xi_k)\to0\hbox{ as }k\to+\infty\ \hbox{ and }\ \psi_k'(\xi_k)\le0\hbox{ for all }k\in\N.
\ee
Up to a shift of the origin, one can assume without loss of generality that
\be\label{xik}
\xi_k=0
\ee
for all $k\in\N$. Up to extraction of a subsequence, one can also assume that $c_k\to c_{\infty}\in[\underline{c},\overline{c}]$ as $k\to+\infty$.

Notice first that Lemma~\ref{lemharnack} and the equations~\eqref{eq:TW} satisfied by $(\phi_k,\psi_k)$ with $c_k\in[\underline{c},\overline{c}]\subset(0,+\infty)$ imply that the sequence $\{\psi'_k/\psi_k\}$ is bounded in $L^{\infty}(\R)$, that is, there is $C>0$ such that $|\psi'_k(\xi)|\le C\,\psi_k(\xi)$ for all $k\in\N$ and $\xi\in\R$. Since $\psi_k(0)\to0^+$ as $k\to+\infty$, it follows that
$$\psi_k\to0\hbox{ locally uniformly in }\R\hbox{ as }k\to+\infty.$$
As a consequence, there also holds that $\psi_k'\to0$ locally uniformly in $\R$ as $k\to+\infty$.

Furthermore, by differentiating the equation~\eqref{eq:TW} satisfied by $\phi_k$, one gets that the functions $\phi'_k$ and $\phi''_k$ are locally bounded (and the functions $\phi_k$ are globally bounded). Therefore, the functions $\phi_k$ converge in $C^1_{loc}(\R)$, up to extraction of a subsequence, to a function $0\le\phi_{\infty}\le1$ solving~\eqref{eq:TW} with speed $c_{\infty}$ and with $\psi=0$, that is,
\be\label{eqphiinfty2}
c_{\infty}\,\phi'_{\infty}=D[\phi_{\infty}]+\mu\,(1-\phi_{\infty})\ \hbox{ in }\R.
\ee
Call $\alpha=\inf_{\R}\phi_{\infty}$ and let $\{\zeta_m\}$ be sequence of real numbers such that $\phi_{\infty}(\zeta_m)\to\alpha$ as $m\to+\infty$. Up to extraction of a subsequence, the functions $\xi\mapsto\phi_{\infty}(\xi+\zeta_m)$ converge as $m\to+\infty$ in $C^{\infty}_{loc}(\R)$ to a function $\Phi_{\infty}$ solving $c_{\infty}\,\Phi_{\infty}'=D[\Phi_{\infty}]+\mu\,(1-\Phi_{\infty})$ in $\R$, $\alpha\le\Phi_{\infty}\le1$ in $\R$ and $\Phi_{\infty}(0)=\alpha$. Consequently, $\Phi_{\infty}'(0)=0$ and $D[\Phi_{\infty}](0)\ge0$, whence $\mu\,(1-\alpha)=\mu\,(1-\Phi_{\infty}(0))\le0$. Thus, $\alpha\ge1$. Since $\alpha=\inf_{\R}\phi_{\infty}$ and $\phi_{\infty}\le1$ in $\R$, one concludes that
$$\phi_{\infty}=1\ \hbox{ in }\R.$$

Now set
$$\Psi_k(\xi)=\frac{\psi_k(\xi)}{\psi_k(0)}$$
for $k\in\N$ and $\xi\in\R$. Since the sequence $\{\psi'_k/\psi_k\}$ is bounded in $L^{\infty}(\R)$, the positive functions $\Psi_k$ are locally bounded, in the sense that $\sup_{k\in\N,\,|\xi|\le R}\Psi_k(\xi)<+\infty$ for all $R>0$. Therefore, the functions
$$\Psi'_k(\xi)=\frac{\psi'_k(\xi)}{\psi_k(0)}=\frac{\psi'_k(\xi)}{\psi_k(\xi)}\times\Psi_k(\xi)$$
are locally bounded too. Since each $\Psi_k$ satisfies
$$-c_k\,\Psi_k''(\xi)+d\,D[\Psi_k'](\xi)-(\mu+\gamma)\,\Psi_k'(\xi)+\beta\,\phi'_k(\xi)\,\Psi_k(\xi)+\beta\,\phi_k(\xi)\,\Psi_k'(\xi)=0$$
in $\R$ and the sequence $\{\phi_k\}$ is bounded in $C^1_{loc}(\R)$, one infers that the functions $\Psi_k''$ are locally bounded too. By Arzela-Ascoli theorem, it follows that, up to extraction of a subsequence, the positive functions $\Psi_k$ converge in $C^1_{loc}(\R)$ to a nonnegative solution $\Psi_{\infty}$ of
\be\label{eqPsiinfty}
c_{\infty}\,\Psi_{\infty}'=d\,D[\Psi_{\infty}]+(\beta-\mu-\gamma)\,\Psi_{\infty}\ \hbox{ in }\R,
\ee
where one used the fact that $\phi_k(\xi)\to\phi_{\infty}(\xi)=1$ as $k\to+\infty$ for all $\xi\in\R$. Furthermore, we claim that $\Psi_{\infty}>0$ in $\R$. Otherwise, there is $\xi_0\in\R$ such that $\Psi_{\infty}(\xi_0)=0$, and $\Psi_{\infty}'(\xi_0)=0$. It follows from~\eqref{eqPsiinfty} applied at $\xi_0$ that $\Psi_{\infty}(\xi_0+1)=\Psi_{\infty}(\xi_0-1)=0$, and then $\Psi_{\infty}(\xi_0+m)=0$ for all $m\in\mathbb{Z}$ by immediate induction. Since $c_{\infty}\Psi_{\infty}'\ge(\beta-\mu-\gamma-2)\,\Psi_{\infty}$ in $\R$, the nonnegative function $\xi\mapsto\Psi_{\infty}(\xi)\,e^{-(\beta-\mu-\gamma-2)\xi/c_{\infty}}$ is nondecreasing. Since it vanishes at $\xi_0+m$ for all $m\in\mathbb{Z}$, one concludes that it is identically equal to $0$, whence $\Psi_{\infty}=0$ in $\R$. This contradicts the fact that $\Psi_{\infty}(0)=1$. Therefore,
$$\Psi_{\infty}(\xi)>0$$
for all $\xi\in\R$.

The continuous function $z:=\Psi_{\infty}'/\Psi_{\infty}$ obeys
\be\label{eqz2}
\frac{c_{\infty}}{d}\,z(\xi)=e^{\int_{\xi}^{\xi+1}z(s)ds}+e^{\int_{\xi}^{\xi-1}z(s)ds}-2+\frac{\beta-\mu-\gamma}{d}\ \hbox{ in }\R.
\ee
Therefore, by Proposition~\ref{key}, $z(\xi)=\Psi_{\infty}'(\xi)/\Psi_{\infty}(\xi)$ has finite limits $\omega_{\pm}$ as $\xi\to\pm\infty$, which are roots of the characteristic equation
$$c_{\infty}\,\omega_{\pm}=d\,(e^{\omega_{\pm}}+e^{-\omega_{\pm}}-2)+\beta-\mu-\gamma.$$
Since $c_{\infty}\ge\underline{c}>0$ and $\beta>\mu+\gamma$, the roots of the previous equation are necessarily positive. In particular, $\Psi_{\infty}'$ is positive at $\pm\infty$. Furthermore, by differentiating~\eqref{eqz2}, one gets that
\be\label{eqz'}
c_{\infty}\,z'(\xi)=d\,(z(\xi+1)-z(\xi))\,\frac{\Psi_{\infty}(\xi+1)}{\Psi_{\infty}(\xi)}+d\,(z(\xi-1)-z(\xi))\,\frac{\Psi_{\infty}(\xi-1)}{\Psi_{\infty}(\xi)}\ \hbox{ in }\R.
\ee
Therefore, if $z$ has a minimum $\underline{\xi}$ in $\R$, then $z'(\underline{\xi})=0$ and $z(\underline{\xi}+1)=z(\underline{\xi}-1)=z(\underline{\xi})$, whence $z(\underline{\xi}+m)=z(\underline{\xi})$ for all $m\in\mathbb{Z}$ by immediate induction. As a consequence,
$$\inf_{\R}z\ge\min\{z(-\infty),z(+\infty)\}>0.$$
Finally, $\Psi_{\infty}'>0$ in $\R$, hence $0<\Psi_{\infty}'(0)=\lim_{k\to+\infty}\Psi_k'(0)=\lim_{k\to+\infty}\psi'_k(0)/\psi_k(0)$ and $\psi'_k(0)>0$ for all $k$ large enough. This contradicts the fact that $\psi'_k(0)\le0$ for all $k\in\N$ (remember~\eqref{psik} and~\eqref{xik}).

As a conclusion, there is $\epsilon>0$ such that $\psi'(\xi)>0$ for any $\xi\in\R$ with $\psi(\xi)\le\epsilon$ for any solution $(\phi,\psi)$ of~\eqref{eq:TW} with $c\in[\underline{c},\overline{c}]$, $0<\phi<1$ and $\psi>0$ in $\R$. The proof of Lemma~\ref{lempsi'} in thereby complete.\hfill$\Box$


\section{The case $c=c^*$}\label{secc*}
\setcounter{equation}{0}

This section is devoted to the proof of the existence of a traveling wave $(\phi,\psi)$ of~\eqref{eq:TW} satisfying~\eqref{nontrivial},~\eqref{bc-minus} and~\eqref{bc-plus} with speed $c=c^*$. To do so, we consider a sequence $\{c_k\}$ of real numbers such that $c_k\in(c^*,c^*+1]$ for each $k\in\N$, and
$$c_k\to c^*\ \hbox{ as }k\to+\infty.$$
For each $k\in\N$, Section~\ref{sec3} provides the existence of a traveling wave $(\phi_k,\psi_k)$ of~\eqref{eq:TW} (with speed $c_k$) satisfying~\eqref{nontrivial},~\eqref{bc-minus} and~\eqref{bc-plus}. The natural strategy is to pass to the limit as $k\to+\infty$, in order to get the existence of a traveling wave with the limiting speed $c^*$. To achieve this goal, we need some a priori bounds for the functions $\psi_k$ in order to get a non-trivial solution at the limit. We also point out that the inequalities~\eqref{ineqphipsi} satisfied by the approximated waves $(\phi_k,\psi_k)$ do not carry over at the limit $c_k\to c^*$ (since the coefficients in the definitions of the lower solutions depend on $c_k$ and degenerate at the limit $c_k\to c^*$). Therefore, we will have to suitably shift and renormalize the approximated waves $(\phi_k,\psi_k)$ before passing to the limit as $k\to+\infty$.

The first a priori bound asserts that the functions $\psi_k$ do not converge to $0$ uniformly as $k\to+\infty$.

\begin{lemma}\label{lemnontrivial}
There holds $\liminf_{k\to+\infty}\|\psi_k\|_{L^{\infty}(\R)}>0$.
\end{lemma}

\begin{proof}
Assume that the conclusion does not hold. Then, up to extraction of a subsequence, one can assume without loss of generality that $\|\psi_k\|_{L^{\infty}(\R)}\to0$ as $k\to+\infty$. Since $c_k\in[c^*,c^*+1]\subset(0,+\infty)$ for each $k\in\N$, Lemma~\ref{lempsi'} implies that $\psi_k'>0$ in $\R$ for all $k$ large enough. Since each $\psi_k$ is bounded, it follows that the limit $\psi_k(+\infty)$ exists in $\R$, for all $k$ large enough. Since each $(\phi_k,\psi_k)$ satisfies the assumptions of Lemma~\ref{lemendemic}, one then infers in particular that, for all $k$ large enough,
$$\psi_k(+\infty)=e^*=\frac{\mu}{\beta}\,(\sigma-1)>0.$$
This contradicts the fact that $\lim_{k\to+\infty}\|\psi_k\|_{L^{\infty}(\R)}=0$. Thus, the conclusion of Lemma~\ref{lemnontrivial} holds.
\end{proof}

The second key-point is the boundedness of the sequence $\{\psi_k\}$ in $L^{\infty}(\R)$.

\begin{lemma}\label{lembound}
There holds $\limsup_{k\to+\infty}\|\psi_k\|_{L^{\infty}(\R)}<+\infty$.
\end{lemma}

\begin{proof}
Assume that the conclusion does not hold. Then, up to extraction of a subsequence, one has $\|\psi_k\|_{L^{\infty}(\R)}\to+\infty$ as $k\to+\infty$. For each $k\in\N$, since the function $\psi_k$ is bounded and positive in $\R$, there is then $\xi_k\in\R$ such that
\be\label{defxik}
\psi_k(\xi_k)\ge\Big(1-\frac{1}{k+1}\Big)\,\|\psi_k\|_{L^{\infty}(\R)}.
\ee
In particular, $\psi_k(\xi_k)\to+\infty$ as $k\to+\infty$. Furthermore, one has
$$\psi_k'(\xi)\ge\frac{d}{c^*}\,\psi_k(\xi+1)-\frac{2d+\mu+\gamma}{c^*}\,\psi_k(\xi)\ \hbox{ in }\R$$
for all $k\in\N$. Since each $\psi_k$ is positive, it follows from Lemma~\ref{lemharnack} that the functions $\xi\mapsto\psi_k(\xi\pm 1)/\psi_k(\xi)$ are globally bounded in $\R$ independently of $k\in\N$, and so are the functions $\xi\mapsto\psi_k'(\xi)/\psi_k(\xi)$, from the equation~\eqref{eq:TW} satisfied with speed $c_k\in(c^*,c^*+1]$ (remember also that $0<\phi_k<1$ in $\R$). As a consequence,
$$\psi_k(\xi+\xi_k)\mathop{\longrightarrow}_{k\to+\infty}+\infty\ \hbox{ locally uniformly in }\xi\in\R.$$
Lemma~\ref{leminfty} then implies that
$$\Phi_k(\xi):=\phi_k(\xi+\xi_k)\to0$$
as $k\to+\infty$ locally uniformly in $\xi\in\R$.

From the boundedness of the sequence $\{\psi_k'/\psi_k\}$ in $L^{\infty}(\R)$, one also infers that the functions
$$\xi\mapsto\Psi_k(\xi)=\frac{\psi_k(\xi+\xi_k)}{\psi_k(\xi_k)}$$
are locally bounded independently of $k$ (in the sense that $\sup_{k\in\N}\|\Psi_k\|_{L^{\infty}(K)}<+\infty$ for any compact set $K\subset\R$). Each function $\Psi_k$ obeys
$$c_k\,\Psi_k'=d\,D[\Psi_k]-(\mu+\gamma)\,\Psi_k+\beta\,\Phi_k\,\Psi_k\ \hbox{ in }\R,$$
whence the functions $\Psi_k'$ are locally bounded too. From Arzela-Ascoli theorem, the positive functions $\Psi_k$ converge locally uniformly in $\R$, up to extraction of a subsequence, to a continuous nonnegative function $\Psi_{\infty}$. Furthermore, from the above equation and the fact that $\Phi_k\to0$ as $k\to+\infty$ locally uniformly in $\R$ (together with $c_k\to c^*>0$), the functions $\Psi_k'$ converge locally uniformly in $\R$ too. Therefore, the functions $\Psi_k$ converge in $C^1_{loc}(\R)$ to $\Psi_{\infty}$ and the function $\Psi_{\infty}$ satisfies
\be\label{eqpsiinfty2}
c^*\,\Psi_{\infty}'=d\,D[\Psi_{\infty}]-(\mu+\gamma)\,\Psi_{\infty}\ \hbox{ in }\R.
\ee
Notice that this function $\Psi_{\infty}$ is thus automatically of class $C^{\infty}(\R)$. Furthermore, $\Psi_{\infty}$ is nonnegative and $\Psi_{\infty}(0)=\lim_{k\to+\infty}\Psi_k(0)=1$. As in the proof of Lemma~\ref{lempsi'} for the solution of~\eqref{eqPsiinfty}, one then infers that $\Psi_{\infty}$ is positive in $\R$.

Finally, for every $\xi\in\R$, there holds $\psi_k(\xi+\xi_k)\le\|\psi_k\|_{L^{\infty}(\R)}\le(1+1/k)\,\psi_k(\xi_k)$ from~\eqref{defxik}. In other words, $\Psi_k(\xi)\le1+1/k$ for every $\xi\in\R$ and $k\in\N$ with $k\ge1$, whence $\Psi_{\infty}(\xi)\le1$ for every $\xi\in\R$. Therefore, since $\Psi_{\infty}(0)=1$, $0$ is a global maximum of the function $\Psi_{\infty}$, and $\Psi_{\infty}'(0)=0$, $D[\Psi_{\infty}](0)\le0$. The equation~\eqref{eqpsiinfty2} evaluated at $0$ leads to a contradiction, since $\mu$ and $\gamma$ are positive. The proof of Lemma~\ref{lembound} is thereby complete.
\end{proof}

\noindent{\it{End of the proof of Theorem~$\ref{th:exist}$ in case $c=c^*$.}} First of all, Lemma~\ref{lempsi'} applied with $\underline{c}=c^*>0$ and $\overline{c}=c^*+1$ yields the existence of $\epsilon>0$ such that $\Psi'(\xi)>0$ for every $\xi\in\R$ with $\Psi(\xi)\le\epsilon$, and for every solution $(\Phi,\Psi)$ of~\eqref{eq:TW} and~\eqref{nontrivial} with speed $c\in[c^*,c^*+1]$. Without loss of generality, one can assume that
\be\label{choiceeps}
0<\epsilon\le e^*=\frac{\mu}{\beta}\,(\sigma-1).
\ee
Coming back to our solutions $(\phi_k,\psi_k)$ of~\eqref{eq:TW} (with speed $c_k$) satisfying~\eqref{nontrivial},~\eqref{bc-minus} and~\eqref{bc-plus}, it follows from Lemma~\ref{lemnontrivial} and the positivity of each $\psi_k$ that one can also assume without loss of generality that
$$0<\epsilon<\inf_{k\in\N}\|\psi_k\|_{L^{\infty}(\R)}.$$
Therefore, for each $k\in\N$, since $\psi_k(-\infty)=0$ and $\psi_k>0$, there is $\xi_k\in\R$ such that
$$\psi_k(\xi_k)=\epsilon.$$
Shift the origin at $\xi_k$ and denote
$$\tilde{\phi}_k(\xi)=\phi_k(\xi+\xi_k)\ \hbox{ and }\ \tilde{\psi}_k(\xi)=\psi_k(\xi+\xi_k).$$
From Lemma~\ref{lembound}, the sequence $\{\tilde{\psi}_k\}$ is bounded in $L^{\infty}(\R)$. Remember also that $0<\tilde{\phi}_k<1$ in $\R$ and $c_k\to c^*>0$ as $k\to+\infty$. Therefore, up to extraction of a subsequence, the functions $\tilde{\phi}_k$ and $\tilde{\psi}_k$ converge in $C^{\infty}_{loc}(\R)$ to some bounded $C^{\infty}(\R)$ functions $\phi$ and $\psi$ solving~\eqref{eq:TW} with speed $c^*$. Furthermore, $0\le\phi\le1$ and $\psi\ge0$ in $\R$, while
$$\psi(0)=\epsilon>0.$$

In order to complete the proof of Theorem~\ref{th:exist} in case $c=c^*$, one shall show that the pair $(\phi,\psi)$ is non-trivial and satisfies the desired limiting conditions at $\pm\infty$, that is, the conditions~\eqref{nontrivial},~\eqref{bc-minus} and~\eqref{bc-plus} hold.

Let us first show that
$$\psi>0\ \hbox{ in }\R.$$
Indeed, if there is $\xi^*\in\R$ such that $\psi(\xi^*)=0$, then $\psi'(\xi^*)=0$ and equation~\eqref{eq:TW} at $\xi^*$ yields $\psi(\xi^*\pm1)=0$, whence $\psi(\xi^*+m)=0$ for all $m\in\mathbb{Z}$ by immediate induction. But $c^*\psi'\ge-(2d+\mu+\gamma)\,\psi$ in $\R$, whence the function $\xi\mapsto\psi(\xi)\,e^{(2d+\mu+\gamma)\xi/c^*}$ is nondecreasing. Since $\psi\ge0$ in $\R$ and $\psi(\xi^*+m)=0$ for all $m\in\mathbb{Z}$, one infers that $\psi=0$ in $\R$, which is impossible since $\psi(0)=\epsilon>0$. Thus, $\psi>0$ in $\R$. Once the positivity of $\psi$ is known, it follows as in the proof of Lemma~\ref{lemnontrivial0} that
$$0<\phi<1\ \hbox{ in }\R.$$
In other words, the pair $(\phi,\psi)$ fulfills~\eqref{nontrivial}.

Let us then show that the pair $(\phi,\psi)$ satisfies the limiting conditions~\eqref{bc-minus} at $-\infty$. Since the pair $(\phi,\psi)$ solves~\eqref{eq:TW} and~\eqref{nontrivial} with speed $c^*$, the choice of $\epsilon>0$ above and the property $\psi(0)=\epsilon$ imply that $\psi'>0$ in $(-\infty,0]$. In particular, the limit $L=\lim_{\xi\to-\infty}\psi(\xi)$ exists, and $L\in[0,\epsilon)$. If $L>0$, then the same arguments as in the proof of Lemma~\ref{lemendemic} imply that $\phi(-\infty)$ exists and $\phi(-\infty)=(\mu+\gamma)/\beta=1/\sigma\in(0,1)$. The same arguments also yield $L=\psi(-\infty)=\mu(1-\phi(-\infty))/(\beta\phi(-\infty))=(\mu/\beta)(\sigma-1)=e^*$. Hence, $e^*=L<\epsilon$, contradicting~\eqref{choiceeps}. Therefore,
$$L=\psi(-\infty)=0.$$
Furthermore, for any sequence $\{\tilde{\xi}_k\}$ converging to $-\infty$, the functions $\xi\mapsto\phi(\xi+\tilde{\xi}_k)$ and $\xi\mapsto\psi(\xi+\tilde{\xi}_k)$ converge in $C^{\infty}_{loc}(\R)$, up to extraction of a subsequence, to a pair $(\phi_{\infty},0)$, for some function $0\le\phi_{\infty}\le1$ solving~\eqref{eqphiinfty2} with speed $c_{\infty}=c^*$. It follows as in the proof of Lemma~\ref{lempsi'} that $\phi_{\infty}=1$ in $\R$. Since the limit does not depend on the choice the sequence $\{\tilde{\xi}_k\}$, one gets that the limit $\lim_{\xi\to-\infty}\phi(\xi)$ exists, and
$$\phi(-\infty)=1.$$
In other words, the pair $(\phi,\psi)$ satisfies~\eqref{bc-minus}.

Let us finally show that the non-triviality conditions~\eqref{bc-plus} hold at $+\infty$. Firstly, as in the proof of Lemma~\ref{leminfphi}, there holds $\inf_{\R}\phi>0$. Secondly, Lemma~\ref{lempsi'} and~\eqref{nontrivial} imply at once that $\liminf_{\xi\to+\infty}\psi(\xi)>0$. Thirdly, one concludes that $\limsup_{\xi\to+\infty}\phi(\xi)<1$ as in the proof of~\eqref{claimphi} and that~\eqref{liminfsup} holds as in the case $c>c^*$. The solution $(\phi,\psi)$ thus fulfills all desired properties and the proof of Theorem~\ref{th:exist} in case $c=c^*$ is thereby complete.\hfill$\Box$


\section{Non-existence of traveling waves for $c<c^*$}\label{sec5}
\setcounter{equation}{0}

In this section, $(\phi,\psi)$ denotes a classical solution of~\eqref{eq:TW} satisfying~\eqref{nontrivial} and~\eqref{bc-minus}, with a speed $c\in\R$. By classical, we mean that $\phi$ and $\psi$ are of class $C^1(\R)$ (and then of class $C^{\infty}(\R)$) if $c\neq 0$, and that $\phi$ and $\psi$ are continuous if $c=0$. We shall prove that, necessarily, $c\ge c^*$. To do so, we consider separately the cases $c>0$, $c<0$ and $c=0$.

{\it First case: $c>0$.} Since the positive function $\psi$ satisfies
$$\psi'(\xi)\ge\frac{d}{c}\,\psi(\xi+1)-\frac{2d+\mu+\gamma}{c}\,\psi(\xi)$$
for all $\xi\in\R$, Lemma~\ref{lemharnack} implies that the functions $\xi\mapsto\psi(\xi\pm1)/\psi(\xi)$ are bounded, and then so is the function $\xi\mapsto\psi'(\xi)/\psi(\xi)$. Consider now any sequence $\{\xi_k\}$ converging to $-\infty$. The positive functions
$$\xi\mapsto\psi_k(\xi):=\frac{\psi(\xi+\xi_k)}{\psi(\xi_k)}$$
are locally bounded and they satisfy
$$c\,\psi'_k=d\,D[\psi_k]-(\mu+\gamma)\,\psi_k+\beta\,\phi(\cdot+\xi_k)\,\psi_k\ \hbox{ in }\R.$$
Therefore, the functions $\psi'_k$ are locally bounded too (remember that $\phi(-\infty)=1$). From Arzela-Ascoli theorem, up to extraction of a subsequence, the functions $\psi_k$ converge locally uniformly (and then in $C^1_{loc}(\R)$ from the above equation) to a function $\psi_{\infty}$ solving
\be\label{eqpsiinfty3}
c\,\psi_{\infty}'=d\,D[\psi_{\infty}]+(\beta-\mu-\gamma)\,\psi_{\infty}\ \hbox{ in }\R.
\ee
Furthermore, $\psi_{\infty}\ge0$ in $\R$ and $\psi_{\infty}(0)=1$. As in the proof of Lemma~\ref{lempsi'} for the function~$\Psi_{\infty}$ solving~\eqref{eqPsiinfty}, it follows that $\psi_{\infty}>0$ in $\R$. Now, the function $z=\psi_{\infty}'/\psi_{\infty}$ solves
\be\label{eqz3}
\frac{c}{d}\,z(x)=e^{\int_x^{x+1}z(s)ds}+e^{\int_x^{x-1}z(s)ds}-2+\frac{\beta-\mu-\gamma}{d}\ \hbox{ in }\R.
\ee
Proposition~\ref{key} implies that the limits $z(\pm\infty)$ exist in $\R$ and are roots $\omega$ of the equation
$$c\,\omega=d\,(e^{\omega}+e^{-\omega}-2)+\beta-\mu-\gamma.$$
Since $c>0$ and $\beta>\mu+\gamma$, the roots must be positive and $c\ge c^*$ by definition of~$c^*$ in~\eqref{c*}. The proof of the necessity condition is thereby complete in the case $c>0$.

{\it Second case: $c<0$.} Denote $\Phi(\xi)=\phi(-\xi)$ and $\Psi(\xi)=\psi(-\xi)$. The functions $\Phi$ and $\Psi$ satisfy~\eqref{eq:TW} and~\eqref{nontrivial} with speed $|c|>0$, together with the limiting conditions $(\Phi(+\infty),\Psi(+\infty))=(1,0)$. Furthermore, since the positive function $\Psi$ satisfies
$$\Psi'(\xi)\ge\frac{d}{|c|}\,\Psi(\xi+1)-\frac{2d+\mu+\gamma}{|c|}\,\Psi(\xi)$$
and~\eqref{eq:TW} with spped $|c|$, it follows as in the above case $c>0$ that the function $\Psi'/\Psi$ is bounded. Since $\Psi>0$ in $\R$ and $\Psi(+\infty)=0$, one can consider a sequence $\{\xi_k\}$ converging to $+\infty$ such that
$$\Psi'(\xi_k)\le0\ \hbox{ for all }k\in\N.$$
As above, up to extraction of a subsequence, the functions $\xi\mapsto\Psi_k(\xi):=\Psi(\xi+\xi_k)/\Psi(\xi_k)$ converge in $C^1_{loc}(\R)$ to a positive solution $\Psi_{\infty}$ of~\eqref{eqpsiinfty3} with $|c|$ instead of $c$, and such that $\Psi_{\infty}(0)=1$. Furthermore, here, $\Psi_{\infty}'(0)\le0$. The function $Z:=\Psi_{\infty}'/\Psi_{\infty}$ satisfies~\eqref{eqz3} with $|c|$ instead of $c$ and it follows from Proposition~\ref{key} that the limits $Z(\pm\infty)$ exist in $\R$ and are roots $\Omega$ of the equation
$$|c|\,\Omega=d\,(e^{\Omega}+e^{-\Omega}-2)+\beta-\mu-\gamma.$$
Since $\beta>\mu+\gamma$, the roots are positive (and $|c|\ge c^*$). In particular, $Z$ is positive at $\pm\infty$. But $Z(0)=\Psi_{\infty}'(0)/\Psi_{\infty}(0)=\Psi_{\infty}'(0)\le0$. Hence, the continuous function $Z$ has a minimum $\Xi$ in $\R$, that is $Z(\Xi)\le Z(\xi)$ for all $\xi\in\R$. By differentiating the equation satisfied by $Z$, one gets as in~\eqref{eqz'} that
$$|c|\,Z'(\xi)=d\,(Z(\xi+1)-Z(\xi))\,\frac{\Psi_{\infty}(\xi+1)}{\Psi_{\infty}(\xi)}+d\,(Z(\xi-1)-Z(\xi))\,\frac{\Psi_{\infty}(\xi-1)}{\Psi_{\infty}(\xi)}\ \hbox{ in }\R.$$
Hence, $Z(\Xi\pm1)=Z(\Xi)$, and $Z(\Xi+m)=Z(\Xi)=\min_{\R}Z$ for all $m\in\mathbb{Z}$ by immediate induction. Therefore, $Z(\pm\infty)=\min_{\R}Z\le Z(0)\le0$, a contradiction with the positivity of $Z(\pm\infty)$. As a consequence, the case $c<0$ is ruled out.

{\it Third case: $c=0$.} Here, the function $\psi$ satisfies $d\,D[\psi]+(\beta\,\phi-\mu-\gamma)\,\psi=0$ in $\R$. Since $d>0$, $\beta>\mu+\gamma$, $\phi(-\infty)=1$ and $\psi>0$ in $\R$, it follows that there exists $\xi_0\in\R$ such that $D[\psi](\xi)<0$ for all $\xi\le\xi_0$. Denote
$$\theta(\xi)=\psi(\xi)-\psi(\xi+1).$$
The condition $D[\psi]<0$ in $(-\infty,\xi_0]$ means that $\theta(\xi-1)<\theta(\xi)$ for all $\xi\le\xi_0$. Furthermore, since $\psi>0$ in $\R$ and $\psi(-\infty)=0$, there is $\xi_1\le\xi_0$ such that $\theta(\xi_1)<0$. Since $\theta(\xi_1-m)<\theta(\xi_1)$ for all $m\in\N$ with $m\ge1$, one infers that
$$\psi(\xi_1-m)-\psi(\xi_1)=\sum_{j=1}^{m}\theta(\xi_1-j)<m\,\theta(\xi_1)$$
for all $m\in\N$ with $m\ge1$. Thus, $\psi(\xi_1-m)<\psi(\xi_1)+m\,\theta(\xi_1)\to-\infty$ as $m\to+\infty$ since $\theta(\xi_1)<0$. This contradicts the positivity of $\psi$. As a consequence, the case $c=0$ is ruled out too and the proof of Theorem~\ref{th:exist} is thereby complete.\hfill$\Box$

\begin{remark}
We give here another proof of the positivity of $c$ when $\psi$ is bounded (cf.~\cite{GW3}). Since $\phi(-\infty)=1$ and $\beta>\mu+\gamma$, there is a sufficiently large $K$ such that
$$\beta\,\phi(\xi)-\mu-\gamma>\frac{\beta-\mu-\gamma}{2}>0\quad\mbox{for $\xi\in(-\infty,-K)$.}$$
Integrating the second equation of \eqref{eq:TW} from $-\infty$ to $\xi<-K$, using $\psi(-\infty)=0$ and the positivity and boundedness of $\psi$, we obtain
\beaa
c\,\psi(\xi)&=&d\left\{\int_{\xi}^{\xi+1}\psi(s)ds-\int_{\xi-1}^{\xi}\psi(s)ds\right\}+\int_{-\infty}^\xi [\beta\,\phi(s)-\mu-\gamma]\,\psi(s)\,ds\\
&\ge&-d\,\{\sup_{s\in\bR}\psi(s)\}+\frac{\beta-\mu-\gamma}{2}\int_{-\infty}^\xi \psi(s)\,ds
\eeaa
for all $\xi<-K$. It follows that the integral
\beaa
R(\xi):=\int_{-\infty}^\xi \psi(s)\,ds
\eeaa
is well-defined for all $\xi<-K$ (and then for all $\xi\in\R$ by continuity of $\psi$). Integrating the second equation of~\eqref{eq:TW} twice, we obtain
\beaa
c\,R(x)=d\left\{\int_{x}^{x+1}R(\xi)\,d\xi-\int_{x-1}^{x}R(\xi)\,d\xi\right\}+\int_{-\infty}^x\int_{-\infty}^\xi [\beta\,\phi(s)-\mu-\gamma]\,\psi(s)\,ds\,d\xi
\eeaa
for all $x\in\R$. Since $R(\xi)$ is strictly increasing, we conclude that $c>0$.
\end{remark}


\end{document}